\documentclass[a4paper,11pt]{article}

\usepackage[T1]{fontenc}
\usepackage{lmodern}

\usepackage{dsfont}

\usepackage[latin1]{inputenc}
\usepackage{amsmath}
\usepackage{amsthm}
\usepackage{amssymb}
\usepackage{mathrsfs}
\usepackage{graphicx}
\usepackage[all]{xy}
\usepackage{hyperref}

\usepackage{makeidx}

\usepackage{stmaryrd}
\usepackage{caption}

\usepackage{abstract} 

\usepackage{cleveref}

\usepackage{xcolor}

\usepackage[toc,page]{appendix}

\newtheorem{thm}{Theorem}[section]
\newtheorem{cor}[thm]{Corollary}
\newtheorem{claim}[thm]{Claim}

\newtheorem{lemma}[thm]{Lemma}
\newtheorem{prop}[thm]{Proposition}

\theoremstyle{definition}
\newtheorem{definition}[thm]{Definition}
\newtheorem{ex}[thm]{Example}

\newtheorem{question}[thm]{Question}

\newcommand{\cut}{\textup{cut}^{\delta}}
\newcommand{\sh}{\textup{sh}}

\definecolor{oussamacomment}{rgb}{0.8,0.33,0}

\def\rquotient#1#2{%
	\makeatletter
	\raise.3ex\hbox{$#1$}/\lower.3ex\hbox{$#2$}%
	\makeatother
}	

\makeatletter
\newcommand{\subjclass}[2][2010]{%
	\let\@oldtitle\@title%
	\gdef\@title{\@oldtitle\footnotetext{#1 \emph{Mathematics subject classification.} #2}}%
}
\newcommand{\keywords}[1]{%
	\let\@@oldtitle\@title%
	\gdef\@title{\@@oldtitle\footnotetext{\emph{Key words and phrases.} #1.}}%
}
\makeatother

\newcommand{\Address}{{
		\bigskip
		\small
		
\noindent\textsc{UCLouvain\\ 
Institut de recherche en Math\'ematiques et physique\\
 Chemin du Cyclotron 2\\
1348 Louvain-la-Neuve (Belgium)}\par\nopagebreak
\noindent\textit{E-mail address}: \texttt{oussama.bensaid@uclouvain.be}
  
  	\bigskip
		\small
		
\noindent\textsc{University of Montpellier\\ 
Institut Math\'ematiques Alexander Grothendieck\\
Place Eug\`ene Bataillon\\
34090 Montpellier (France)}\par\nopagebreak
\noindent\textit{E-mail address}: \texttt{anthony.genevois@umontpellier.fr}
  
    \bigskip
		\small
		
\noindent\textsc{University of Paris-Cit\'e\\ 
Institut de Math\'ematiques de Jussieu-Paris Rive Gauche\\
Place Aur\'elie Nemours\\
75013 Paris (France)}\par\nopagebreak
\noindent\textit{E-mail address}: \texttt{romain.tessera@imj-prg.fr}

}}

\makeindex

\title{Coarse separation and splittings in hyperbolic groups}
\date{\today}
\author{Oussama Bensaid, Anthony Genevois, and Romain Tessera}


\begin{document}

\maketitle

\begin{abstract}
We study coarse separation in one-ended hyperbolic groups from a quantitative point of view, focusing on the volume growth of separating subsets. We prove that a one-ended hyperbolic group that is not virtually a surface group is coarsely separable by a subset of subexponential growth if and only if it splits over a virtually cyclic subgroup. To do so, we show that sufficiently large thickened spheres are hard to cut, in the sense that their cut-sets have exponential size, a result of independent interest. As an application, we obtain a polynomial lower bound on the separation profile of one-ended hyperbolic groups that do not split over a two-ended subgroup. We also apply our criterion to graph products of finite groups, giving a combinatorial characterisation of when such graph products are coarsely separable by a subset of subexponential growth.
\end{abstract}

\tableofcontents

\section{Introduction}

Separation phenomena in groups have long played a central role in geometric group theory, as they often reflect the presence of algebraic splittings. The starting point is Stallings' theorem \cite{MR0415622}: a finitely generated group has more than one end if and only if it splits over a finite subgroup. Equivalently, a Cayley graph of the group can be separated into at least two deep components by a finite subset if and only if the group splits over a finite subgroup. In this sense, a coarse separation property already detects a nontrivial splitting, and hence an algebraic feature that is invariant under quasi-isometry. This philosophy has since been developed in several directions. A first generalisation of this principle was obtained by Dunwoody--Swenson \cite{MR1760752}, who extended the subgroup-theoretic side of Stallings' theorem from finite subgroups to finitely generated virtually polycyclic codimension-one subgroups. In particular, if $G$ is one-ended and not virtually a surface group, then $G$ splits over a two-ended subgroup if and only if it contains an infinite cyclic codimension-one subgroup. Recall that a subgroup $H \leq G$ is \emph{codimension-one} if some finite neighbourhood of a copy of $H$ in a Cayley graph of $G$ separates that Cayley graph into at least two deep components. Another generalisation was later proved by Papasoglu \cite{MR2153400}, who extended the geometric side of Stallings' theorem by replacing finite separating sets with quasi-lines: he showed that a one-ended finitely presented group that is not commensurable to a surface group splits over a two-ended subgroup if and only if its Cayley graph is separated by a quasi-line. In particular, admitting such a splitting is a quasi-isometry invariant. In the same spirit, and following \cite{bensaid2024coarse}, we study coarse separation by families rather than individual subsets. We do so from a quantitative point of view, focusing on the volume growth of the separators rather than on their quasi-isometry type. We prove

\begin{thm}\label{thm:IntroHyp}
Let $G$ be a hyperbolic group. Assume that $G$ is one-ended and not virtually a surface group. Then $G$ is coarsely separable by a family of subexponential growth if and only if it splits over a virtually cyclic subgroup. 
\end{thm} 

We refer to \Cref{subsec:coarse_sep_cut} for the definitions of coarse separation and of the volume growth of a family of subsets. The exclusion of virtually surface groups is necessary. Indeed, virtually surface groups are quasi-isometric to $\mathbb H^2$, so they are coarsely separable by subsets of linear growth, for instance by thickened geodesics. However, there exist hyperbolic triangle groups that are virtually surface groups but do not split over two-ended subgroups \cite{BowditchCut}. We also note that, in the hyperbolic setting, the previous discussion has a boundary counterpart: Bowditch \cite{BowditchCut} showed that splittings over two-ended subgroups are detected by local cut points in the boundary. This point of view will be essential in the proof.

\paragraph{Coarse separation and splittings.} Theorem~\ref{thm:IntroHyp} exhibits a form of rigidity, proving that, for one-ended hyperbolic groups that are not virtually surface groups, coarse separation by a family of subexponential growth can occur only when there is a strong form of algebraic separation, namely a splitting over a virtually cyclic subgroup. In a forthcoming article \cite{SepRAAGs}, we will highlight a similar phenomenon for right-angled Artin groups by proving that a right-angled Artin group is coarsely separable by a family of subexponential growth if and only if it splits over an abelian subgroup. Then, it is natural to wonder to which extent this observation is common among finitely generated groups.

\begin{question}\label{question:IntroVague}
Under which reasonable assumptions does a finitely generated group coarsely separable by a family of subexponential growth (virtually) split over a subgroup of subexponential growth?
\end{question}

\noindent
In view of Theorem~\ref{thm:IntroHyp} and \cite{SepRAAGs}, it is reasonable to expect positive answers to the following questions:

\begin{question}
If a relatively hyperbolic group is coarsely separable by a family of subexponential growth, does it virtually split over a virtually cyclic subgroup? 
\end{question}

\begin{question}
If a right-angled Coxeter group, or more generally a virtually cocompact special group, is coarsely separable by a family of subexponential growth, does it virtually split over a virtually abelian subgroup?
\end{question}

\noindent
Question~\ref{question:IntroVague} can be decomposed into two subquestions. First, one can ask whether the geometric separation can occur through some subgroup. A priori, there is no reason for a coarsely separating family to be close in any sense to a subgroup. In fact, one can always enrich a coarsely separating family in order to make it as different as we want from a subgroup. Nevertheless, we can ask: if a given finitely generated $G$ is coarsely separable by a family of subexponential growth, can we find a codimension-one subgroup $H \leq G$ of subexponential growth? Note that one may also define a codimension-one subgroup $H \leq G$ as a subgroup such that the Schreier graphs $\mathrm{Sch}(G,H)$, constructed from finite generating sets of $G$, are multi-ended. One also says that $G$ \emph{semi-splits over $H$}.

\medskip \noindent
A degenerate case of our question deals with finitely generated groups that have subexponential growth themselves. Can we always find codimension-one subgroups in such groups? Recall from \cite{MR1347406,MR1459140} that a finitely generated group contains a codimension-one subgroup if and only if it admits an action on a median graph (or equivalently, a CAT(0) cube complex) with unbounded orbits. When such an action does not exist, one says that the group satisfies \emph{Property (FW)}. 

\begin{question}
Does there exist a group of intermediate growth satisfying (FW)?
\end{question}

\noindent
According to \cite{MR1459140}, groups satisfying Kazhdan's Property (T) provide a source of examples of groups satisfying (FW), and in fact the almost exclusive source of such examples. But they are of no help here, since groups of subexponential growth are amenable and that the only amenable groups satisfying (T) are the finite groups. 

\medskip \noindent
Our second subquestion deals with a problem that had attracted a lot of attention a couple of dedaces ago: assuming that a finitely generated group $G$ semi-splits over a subgroup $H$, can we find a (virtual) splitting of $G$ over a subgroup related to $H$? We refer the reader to the survey \cite{MR2031877} for more information on this general problem. Here, we are interested in the case where $H$ has subexponential growth. Notice that, even under this restriction, a positive answer to the question is not reasonable. For instance, it is known that the Grigorchuk group admits codimension-one subgroups, necessarily of subexponential growth, but, since the Grigorchuk group is torsion, there is no (virtual) splitting what so ever. On the other hand, it is worth mentioning that, for polynomial growth (i.e.\ virtually nilpotent groups), the problem is rather well-understood thanks to the algebraic torus theorem \cite{MR1760752} (which, more generally, deals with semi-splittings over polycyclic subgroups). Since the situation seems to be more tame for polynomial growth, it is then natural to ask:

\begin{question}
Does a finitely generated group coarsely separable by a family of polynomial growth virtually split over a subgroup of polynomial growth?
\end{question}

\noindent
In fact, in view of Stalling's theorem about the number of ends and of Papasoglu's theorem \cite{MR2153400} about coarse separation by quasi-lines, it is even more natural to ask:

\begin{question}
Does a finitely generated group coarsely separable by an $n$-dimensional quasiflat virtually split over a virtually abelian subgroup of rank $n$?
\end{question}

\paragraph{Strategy.} In fact, we show that if $G$ is a one-ended hyperbolic group that is not virtually a surface group and does not split over a virtually cyclic subgroup, then there exists a constant $D \geq 0$ such that the $D$-neighbourhoods of all sufficiently large spheres are ``hard to cut'', in the sense that their cut sets have uniformly exponential growth, see \Cref{subsec:coarse_sep_cut} for a definition of the cut of a subgraph. Our main result, from which \Cref{thm:IntroHyp} follows, is the following theorem, which is of independent interest.
\begin{thm}\label{thm:cut_spheres_exp}
    Let $G$ be a one-ended hyperbolic group that is not virtually a surface group and does not split over any two-ended subgroup. Then there exists $t_0 \ge 1$ such that, for any $t \ge t_0$ and any $\varepsilon \in (0,1)$, there exist constants $\lambda > 0$ and $C > 0$ such that the following holds: for every $g \in G$ and every sufficiently large integer $n$,
    $$
        \textup{cut}^\varepsilon\left(S_G(g,n)^{+t}\right) \ge C e^{\lambda n}.
    $$
\end{thm}

A rough outline of the proof is as follows. First, note that, since $G$ is one-ended, a thickened sphere $S_G(g,n)^{+t}$ is a connected subgraph of $G$, for $t$ large enough. Let $Z_n \subset S_G(g,n)^{+t}$ be an $\varepsilon$-cut, and suppose towards contradiction that $Z_n$ is ``small''. A first step shows that $S_G(g,n)^{+t}\setminus Z_n$ must contain two connected components with points $x_n,y_n$ lying very far from $Z_n$. We then pass from the sphere to the boundary by means of shadows: the shadows $\sh(x_n)$ and $\sh(y_n)$ lie in different connected components of $\partial G \setminus \sh(Z_n)$, so that $\sh(Z_n)$ separates the boundary. Moreover, they lie ``far'' from $\sh(Z_n)$. We then use a result of Lazarovich--Margolis--Mj \cite{lazarovich2024commensurated}, see \Cref{thm:LMM}, which provides, inside any minimal closed separator of the boundary, a uniform density property at all small scales. By iterating this density property, we obtain exponentially many well-separated points in $\sh(Z_n)$. Because they are well separated, we can pull this separated family back to a separated family inside $Z_n$, forcing $Z_n$ itself to be ``big'', a contradiction.

\Cref{thm:IntroHyp} follows from \Cref{thm:cut_spheres_exp} by a result from \cite{bensaid2024coarse}, see \Cref{thm:persist_family_expcut_coarse_sep}. The idea is as follows. First, thickened spheres form a \emph{persistent} family, in the sense that two neighbouring thickened spheres intersect in a positive proportion. Therefore, if two thickened spheres are separated by some subset, then, by moving from one sphere to the other through neighbouring spheres, one finds a thickened sphere whose intersection with the separating subset is a cut-set. Since such cut-sets have exponential size by \Cref{thm:cut_spheres_exp}, it follows that the separating subset itself must have exponential growth.

\paragraph{Applications of our techniques.}
As a consequence of \Cref{thm:cut_spheres_exp}, such hyperbolic groups have a separation profile that is bounded below by a polynomial:

\begin{thm}\label{thm:Intro_sep_profile_hyp_groups}
    Let $G$ be a one-ended hyperbolic group that is not virtually a surface group. If $G$ does not split over any two-ended subgroup, then there exist constants $C > 0$ and $\varepsilon > 0$ such that, for every $n \in \mathbb{N}$,
    $$
        \textup{sep}_G(n) \ge C n^\varepsilon.
    $$
\end{thm}
We refer to \Cref{subsec:coarse_sep_cut} for the definition of the separation profile of a group. We note that this result was also obtained recently by Hume--Mackay \cite[Theorem~1.11]{arXiv:2511.10469} using different methods: 
instead of proving that spheres are hard to cut, they use techniques
from \cite{MR2667133} to embed ``sufficiently thick'' round trees.

\medskip\noindent
Another application of \Cref{thm:IntroHyp} concerns graph products of finite groups. Our main result in this direction is the following.

\begin{thm}\label{thm:GPsubsep}
Let $\Gamma$ be a finite $\square$-free graph and $\mathcal{G}$ a collection of finite groups indexed by $V(\Gamma)$. The graph product $\Gamma \mathcal{G}$ is coarsely separable by a family of subexponential growth if and only if $\Gamma$ contains a separating subgraph that can be written as a join $A \ast B$, where $A$ is a complete graph and $B$ is either empty or consists of two non-adjacent vertices labelled by $\mathbb{Z}_2$. 
\end{thm}

A brief outline of the proof is as follows. First, we determine exactly when a graph product is virtually cyclic or virtually a surface group. Then, we characterise splittings over virtually cyclic subgroups in terms of separating subgraphs of the defining graph. Finally, we combine this characterisation with \Cref{thm:IntroHyp} and with the classification of virtually cyclic graph products to obtain the theorem.

\subsection*{Acknowledgements}
The first-named author acknowledges support from the FWO and F.R.S.-FNRS under the Excellence of Science (EOS) programme (project ID 40007542).

\section{Preliminaries}

If $(X,d)$ is a metric space, we denote by $B(x,r)$ (resp.\ $S(x,r)$) the closed ball (resp.\ the sphere) of radius $r$:
$$
B(x,r) := \{ z \in X \mid d(x,z) \leq r \}, \qquad
S(x,r) := \{ z \in X \mid d(x,z) = r \}.
$$
If $A \subseteq X$ is a subset, we denote by $|A| \in \mathbb{N} \cup \{+\infty\}$ its cardinality, and by $A^{+r}$ its $r$-neighbourhood:
$$
A^{+r} := \{ x \in X \mid d(x,A) \leq r \}.
$$

\subsection{Coarse separation and cuts}\label{subsec:coarse_sep_cut}

We start by recalling some definitions from \cite{bensaid2024coarse}. 
Since we will only work with graphs in this paper, we restrict all definitions to graphs. 
We begin with the central notion of coarse separation.

\begin{definition}\label{def:coarse_sep}
Let $X$ be a graph and $\mathcal{Z}$ a collection of subgraphs. A connected subgraph $Y \subset X$ is \emph{coarsely separated by $\mathcal{Z}$} if there exists $L \geq 0$ such that for every $D \geq 0$, there is some $Z \in \mathcal{Z}$ such that $Y \setminus Z^{+L}$ has at least two connected components with points at distance $\geq D$ from $Z$. We will call $L$ the \emph{thickening constant} of $\mathcal{Z}$.
\end{definition}

\begin{definition}\label{def:growth_separating_family}
    Let $X$ be a graph of bounded degree, and let $\mathcal S$ be a family of subgraphs of $X$. We define its growth by
    $$
       V_{\mathcal{S}}(r) := \sup_{s\in Y,\ Y\in \mathcal S}|B_X(s,r)\cap Y|.
    $$
    We say that $\mathcal S$ has \emph{exponential volume growth} if
    $$
        \limsup_{r\to\infty} \frac{1}{r}\log V_{\mathcal S}(r) > 0.
    $$
    Otherwise, we say that $\mathcal S$ has \emph{subexponential volume growth}. We denote by $\mathfrak{M}_{\mathrm{exp}}$ the class of bounded degree graphs such that any coarsely separating family of subgraphs must have exponential volume growth.
\end{definition}

The following definitions of a ``cut'' of a subgraph and of the separation profile of a graph are due to Benjamini--Schramm--Tim{\'a}r \cite{benjamini2012separation} and quantify the connectivity of the graph.

\begin{definition}\label{def:cut}
    Let $X$ be a connected graph of bounded degree, and let $\delta \in (0,1)$. Let $A \subset X$ be a finite subgraph, and let $S \subset A$. Then $S$ is a \emph{$\delta$-cut} of $A$ if every connected component of $A \setminus S$ has size $\leq \delta |A|$. We denote by $\cut(A)$ the minimal size of a $\delta$-cut of $A$. The \emph{separation profile} of $X$ is the function
    $$
        \textup{sep}_X(n) := \sup\{\,\textup{cut}^{\frac{1}{2}}(S) : S \subset X,\ |S| \le n\,\},
        \qquad n \in \mathbb N.
    $$
\end{definition}
In the definition of the separation profile, the constant $1/2$ can be
replaced by any $\delta \in (0,1)$ without changing the function up to the
following equivalence: two functions $f,g : \mathbb{N} \to [0,\infty)$ are
said to be equivalent if there exists a constant $C \ge 1$ such that
$$
    \frac{1}{C} g(n)  \le f(n) \le C g(n) 
    \quad \text{for all } n \in \mathbb{N}.
$$

\begin{lemma}\label{lem:partition_connected_comp}
        Let $X$ be a connected graph of bounded degree, and $\delta \in (0,1)$. Set $\delta' = \min (\frac{\delta}{4}, \frac{1-\delta}{4})$. Then for any subset $E \subset X$ and any $\delta$-cut $S \subset E$ of $E$, one of the following holds
        \begin{itemize}
            \item [(1)] $|S| \geq \delta' |E|;$
            \item [(2)] $E \setminus S$ can be partitioned 
        $$ E \setminus S = A \sqcup B,$$
        where each of $A$ and $B$ is a union of connected components of $E \setminus S$, $|A| \geq \delta' |E|$ and $| B | \geq \delta' |E|$.
        \end{itemize}
\end{lemma}
\begin{proof}
        Suppose $|S| < \delta' |E|$. Let $(A_i)_{i=1, \dots, n}$ be the connected components of $E \setminus S$. So
        $$ \sum_{i=1}^n |A_i| > (1-\delta') |E|.$$
        Suppose that the connected components are ordered such that $|A_i| \geq |A_{i+1}|$. If $|A_1| \geq \delta' |E|$, then 
        $$\sum_{i=2}^n |A_i| \geq (1-\delta') |E| - \delta |E| \geq \delta' |E|,$$
        because $|A_1| \leq \delta |E|$. Therefore we have the partition. Otherwise, let $p$ be the smallest integer such that $\sum_{i=1}^p |A_i| \geq \delta' |E|$. So $\sum_{i=1}^p |A_i| \leq 2 \delta' |E|$. 
        $$\sum_{i=p+1}^n |A_i| \geq (1-\delta') |E| - 2\delta' |E| \geq \delta' |E|,$$
        because $\delta' \leq \frac{1}{4}$.
\end{proof}

\begin{definition}\label{def:persistent_family}
     Let $X$ be a connected graph of bounded degree, and let $r_0 \ge 0$ and $\alpha > 0$. 
     A family $\left\{ A_x(r) \right\}_{x \in X,\, r \ge r_0}$ of subgraphs of $X$ is called \emph{$\alpha$-persistent} if, for every $r \ge r_0$, the following hold:
     \begin{itemize}
         \item for all $x \in X$, $A_x(r) \subset B_X(x, 4r)$;
         \item for all $x,y \in X$, $|A_x(r)| = |A_y(r)|$;
         \item for every pair of neighbouring vertices $x,y \in X$, $|A_x(r) \cap A_y(r)| \ge \alpha |A(r)|$.
     \end{itemize}
\end{definition}

\begin{ex}\label{exple:persistent}
    Let $X$ be a connected vertex-transitive graph of valency $d$, and let $t \ge 1$. For every $x \in X$ and every $r \ge t$, set 
    $$A_x(r) := S(x,r)^{+t}.$$
    Then the family $\{A_x(r)\}_{x \in X,\, r \ge t}$ is $\alpha$-persistent, with $\alpha = \frac{1}{|B(t)|}$, see \cite[Example 2.14]{bensaid2024coarse}.
\end{ex} 
We end this subsection with the following result, which is a special case of \cite[Theorem 3.5]{bensaid2024coarse} and motivates the definition of persistent families.

\begin{thm}\label{thm:persist_family_expcut_coarse_sep}
    Let $X$ be a connected graph of bounded degree, and let $\{A_{x}(r)\}_{x,r}$ be an $\alpha$-persistent family. 
    If there exist constants $C,\beta>0$ such that, for every $x \in X$ and every sufficiently large $r$,
    $$
        \cut(A_x(r)) \geq C e^{\beta r},
    $$
    where $\delta = 1 - \frac{\alpha}{2}$, then $X \in \mathfrak{M}_{\mathrm{exp}}$.
\end{thm}

\subsection{Hyperbolic groups and their boundaries}
\subsubsection{Gromov boundary}\label{subsection:background_prelims}
Let $G$ be a non-elementary hyperbolic group, i.e.\ an infinite hyperbolic group that is not virtually cyclic. We also denote by $G$ its Cayley graph with respect to a fixed finite generating set, and let $\delta$ be its hyperbolicity constant, in the sense that every geodesic triangle in $G$ is $\delta$-slim.
The Gromov boundary $\partial G$ admits a visual metric $\rho$: there exist $a>1$, $C \geq 1$ such that, for all $\xi, \eta \in \partial G$ and for any bi-infinite geodesic $\gamma$ connecting them, we have 
    $$ \frac{1}{C} a^{-d(e,\gamma)} \leq \rho(\xi,\eta) \leq  C a^{-d(e,\gamma)}.$$
Given $x,y \in G$ and $\xi \in \partial G$, we denote by $[x,y]$ a geodesic segment from $x$ to $y$, and by $[x,\xi)$ a geodesic ray from $x$ to $\xi$. The following proposition is standard. We include a proof for completeness.
\begin{prop}\label{prop:dist_boundary_sphere_comparison}
    There exist constants $C_1, C_2 > 0$ such that for any $\xi, \eta \in \partial G$, any $n \in \mathbb{N}$, and any $x,y \in S_n$ for which there exist geodesic rays $[e,\xi)$ and $[e,\eta)$ passing through $x$ and $y$, respectively, we have
    \begin{equation}\label{eq distance boundary from S_n upper bound}
        \rho(\xi,\eta) \leq C_2 \, a^{-n} a^{\frac{d(x,y)}{2}}.
    \end{equation}
    In particular, for every $ \omega > 0$,
    \begin{equation}\label{eq:if_rho_big_d_big}
        \textup{if } \quad \rho(\xi,\eta) \geq \omega   \qquad \text{   then   } \qquad d(x,y) \geq 2n + 2 \log_a \left(\frac{\omega}{C_2}\right).
    \end{equation}
    Moreover, if $d(x,y) \geq d_0 := 20 \delta+1 $, then
    \begin{equation}\label{eq distance boundary from S_n lowel bound}
        C_1 \, a^{-n} a^{\frac{d(x,y)}{2}} \leq \rho(\xi,\eta).
    \end{equation}
    In particular, if $d(x,y) \geq d_0$, then $\rho(\xi_x,\xi_y) >0$ for any $\xi_x \in \partial G$ (resp.\ $\xi_y$) extending $[e,x]$ (resp.\ $[e,y]$). 
\end{prop}
\begin{proof}
    For the first inequality, we refer to \cite[Claim 3.19 (i)]{bensaid2024coarse}. Suppose that $d(x,y) \geq 20\delta+1$, and let $\tilde \gamma$ be a bi-infinite geodesic with endpoints $\xi,\eta$. The (generalised) triangle $[e, \xi] \cup  \tilde \gamma \cup [e,\eta]$ is $5\delta$-slim, see \cite[Proposition 2.2]{CoornaertDelzantPapadopoulos1990}, so $x \in ( \tilde\gamma \cup [e,\eta])^{+5\delta}$. If $x \in  [e,\eta]^{+5\delta} $, we would get $d(x,y) \leq 10 \delta$. Therefore $x \in \tilde \gamma^{+5\delta}$, and similarly $y \in \tilde \gamma^{+5\delta}$. Let $x',y' \in \tilde \gamma$ such that
    $$ d(x,x') \leq 5\delta, \quad d(y,y') \leq 5\delta.$$
    Let $\gamma$ be a geodesic segment from $x$ to $y$, and let $m$ be its midpoint. Let $\alpha$ be a geodesic segment from $x$ to $y'$. In the geodesic triangle $\gamma \cup \alpha \cup [y,y']$, since $d(m,y) \geq 10\delta$, then $d(m, [y,y']) \geq 5\delta$. Therefore, there exists $m_1 \in \alpha$ such that $d(m,m_1) \leq \delta$. Now consider the geodesic triangle $\alpha \cup [x',y'] \cup [x,x']$. We have $d(m_1,x) \geq d(m,x) -\delta \geq 9\delta$, so $d(m_1,[x,x']) \geq 4 \delta$. Hence there exists $m' \in [x',y']$ such that $d(m_1,m') \leq \delta$. Therefore 
    $$d(m,m') \leq 2 \delta.$$
    Now consider the geodesic triangle $[e,x] \cup \gamma \cup [e,y]$. Again by $\delta$-slimness, $m \in ([e,x] \cup [e,y])^{+\delta}$. We may assume that there exists $z \in [e,x]$ such that $d(m,z) \leq \delta$. Since
    $$\frac{d(x,y)}{2} = d(x,m) \leq d(x,z)+d(z,m) \leq d(x,z)+ \delta, $$
    we get that $d(x,z) \geq \frac{d(x,y)}{2} - \delta$. Therefore $d(e,z) \leq n -\frac{d(x,y)}{2} - \delta$, and 
    $$d(e,m)\leq n- \frac{d(x,y)}{2}+ 2 \delta.$$
    Finally, we get
    $$d(e,\tilde \gamma) \leq d(e,m') \leq d(e,m) + d(m,m') \leq n - \frac{d(x,y)}{2}+4 \delta.$$
    Therefore,
    $$\rho(\xi,\eta) \geq \frac{1}{C}a^{-d(e,\tilde\gamma)} \geq \frac{1}{C} a^{-(n-\frac{d(x,y)}{2}+4\delta)} = \frac{a^{-4\delta}}{C}\, a^{-n} a^{\frac{d(x,y)}{2}}.$$
    Set $C_1:= \frac{a^{-4\delta}}{C}$.
\end{proof}
By \cite[Thm 7.2]{coornaert1993mesures}, there exist $\alpha>0$ and $K\geq 1$ such that for any $g \in G$ and any integer $n$,
    \begin{equation}\label{eq exp growth of balls}
        \frac{1}{K} e^{\alpha n}  \leq | B(g,n) | \leq K e^{\alpha n} .
    \end{equation}
    Since $G$ is non-elementary, therefore non-amenable, there exists $K' \geq 1$ such that for any $g \in G$ and any integer $n$,
    \begin{equation}\label{eq exp growth of spheres}
        \frac{1}{K'} e^{\alpha n}  \leq | S(g,n) | \leq K' e^{\alpha n} .
    \end{equation}
To simplify notation, we set $B_n := B(e,n)$ and $S_n := S(e,n)$.
As a consequence of Lemma~\ref{lem:partition_connected_comp}, the next proposition shows that, unless an $\varepsilon$-cut of a thickened sphere is exponentially large, it must leave two connected components containing points very far from the cut.

\begin{prop}\label{prop:big_spheres_farfrom_cuts}
    Let $G$ be a non-elementary hyperbolic group, and let $\alpha, K, K'$ be the constants given by \eqref{eq exp growth of balls} and \eqref{eq exp growth of spheres}. Let $\varepsilon \in (0,1)$ and $t \geq 1$. Then there exists $\Theta > 0$ such that, for every integer $n$ and every $\varepsilon$-cut $Z_n \subset S_n^{+t}$, one of the following holds:
    \begin{itemize}
        \item[(1)] there exist $x_n,y_n$ in different connected components of $S_n^{+t} \setminus Z_n$ such that
        $$
        d(\{x_n,y_n\},Z_n) > \frac{n}{2};
        $$
        \item[(2)] $|Z_n| \geq \Theta e^{\frac{\alpha}{2}n}$.
    \end{itemize}
\end{prop}

\begin{proof}
      Let $n$ be an integer, and let
      $$
      S_n^{+t} \setminus Z_n = A \sqcup B
      $$
      be a partition of the connected components of $S_n^{+t} \setminus Z_n$ as in Lemma~\ref{lem:partition_connected_comp}. Denote the connected components by $A = \sqcup_{i \in I} A_i$ and $B = \sqcup_{j \in J} B_j$. If $A_i \subset Z_n^{+\frac{n}{2}} $ for any $i \in I$, then $A \subset Z_n^{+\frac{n}{2}} $, and
      $$ \varepsilon' |S_n^{+t}| \leq  |A| \leq |Z_n| \times \sup_{g \in G}|B\left(g,n/2\right)| \leq |Z_n| K e^{\alpha \frac{n}{2}},$$
      where $\varepsilon' = \min (\frac{\varepsilon}{4}, \frac{1-\varepsilon}{4})$. Since
      $$    \frac{1}{K'} e^{\alpha n} \leq  |S_n|   \leq  |S_n^{+t}|, $$
      we get
      $$|Z_n| \geq \Theta e^{\frac{\alpha}{2}n},$$ 
      for $\Theta = \frac{\varepsilon'}{K K'}$. Otherwise, let $i_0 \in I$ such that $A_{i_0} \not\subset Z_n^{+\frac{n}{2}}$. If $(1)$ is not satisfied, then for any $j \in J$, $B_j \subset Z_n^{+\frac{n}{2}}$. Therefore $B \subset Z_n^{+\frac{n}{2}}$ and the same argument implies once again that
      $|Z_n| \geq \Theta e^{\frac{\alpha}{2}n}$. 
\end{proof}

For every $n \in \mathbb{N}$, let $ \pi_n : G \setminus B_{n-1} \to S_n $ be the projection that associates to $g$ the intersection of $S_n$ with $[e,g]$, with respect to a fixed choice of geodesic segments $[e,g]$ from $e$.
\begin{prop}\label{prop:basics_hyp_groups1}
Let $G$ be a non-elementary hyperbolic group.
\begin{itemize}
    \item[(i)] For $d_0>0$ as in Proposition~\ref{prop:dist_boundary_sphere_comparison}, if $\gamma_1,\gamma_2$ are asymptotic geodesic rays emanating from $e$, then for any $t \geq 0$,
    $$
    d(\gamma_1(t),\gamma_2(t)) \leq d_0.
    $$
    \item[(ii)] There exists a constant $D \geq 0$ such that every $g \in G$ is at distance at most $D$ from a geodesic ray emanating from $e$.
    \item[(iii)]For every $x,y \in G \setminus B_{n-1}$,
    $$
    d(\pi_n(x),\pi_n(y)) \leq d(x,y) + 2\delta.
    $$
\end{itemize}
\end{prop}

\begin{proof}
Part $(i)$ follows directly from \eqref{eq distance boundary from S_n lowel bound}, and see \cite[Lemma 3.1]{bestvina1991boundary} for $(ii)$.

$(iii)$ Let $x,y \in G \setminus B_{n-1}$, and let $\alpha$ be the geodesic segment $[x,y]$. Consider the geodesic triangle $[e,x] \cup \alpha \cup [y,e]$. If $\pi_n(x) \in [e,y]^{+\delta}$ or $\pi_n(y) \in [e,x]^{+\delta}$, then it is easy to check that
$$
d(\pi_n(x),\pi_n(y)) \leq 2\delta.
$$
Otherwise, let $x',y' \in \alpha$ be such that $d(\pi_n(x),x') \leq \delta$ and $d(\pi_n(y),y') \leq \delta$. Since $d(x',y') \leq d(x,y)$, we conclude by the triangle inequality.
\end{proof}
\begin{prop}\label{prop:basic_hyp_groups2}
Let $G$ be a non-elementary hyperbolic group, $d_0>0$ be as in Proposition~\ref{prop:dist_boundary_sphere_comparison}, $n \in \mathbb{N}$, and $x,y \in S_n$. Then the following holds.
 \begin{itemize}
        \item[(i)] If $x$ lies in some geodesic ray $[e,\xi)$ and $y$ lies in some geodesic ray $[e,\eta)$, then any continuous path $\tilde \gamma \subset \partial G $ from $\xi$ to $\eta$ can be pulled-back to a $d_0$-path in $S_n$: there exists a sequence of points $x=x_0, x_1,\dots,x_k = y$ in $S_n$ such that, for any $i$,  $d(x_i,x_{i+1}) \leq d_0$, and $x_i$ lies in some geodesic ray $[e,\xi_i)$, with $\xi_i \in \tilde{\gamma}$.
        \item[(ii)] If $G$ is one-ended, set
        $$
        t_0 := 3 \max\{D, d_0\},
        $$
        where $D$ is the constant from Proposition~\ref{prop:basics_hyp_groups1} and $d_0$ is the constant from Proposition~\ref{prop:dist_boundary_sphere_comparison}. Then for every $t \ge t_0$ and $n \in \mathbb{N}$, $S_n^{+t}$ is connected.

        \item[(iii)] Let $t \geq t_0$. If there exists a path of length $d$ from $x$ to $y$ in $G \setminus B_{n-1}$, then there exists a path from $x$ to $y$ in $S_n^{+t}$ of length $\leq (1+2\delta)d$. 
    \end{itemize}   
\end{prop}
\begin{proof}
    $(i)$ Let $\tilde \gamma :  [0,L] \to \partial G$ be a continuous path. By compactness of $\tilde{\gamma}([0,L])$, there exists a partition
$$
0 = t_0 \leq t_1 \leq \dots \leq t_k = L
$$
such that for any $i$,
$$
\rho\bigl(\tilde{\gamma}(t_i), \tilde{\gamma}(t_{i+1})\bigr) < \frac{1}{C_1} a^{-n},
$$
where $C_1$ is the constant from Proposition~\ref{prop:dist_boundary_sphere_comparison}. For each $i$, let $x_i$ be the intersection of a geodesic ray from $e$ to $\tilde{\gamma}(t_i)$ with $S_n$. By \eqref{eq distance boundary from S_n lowel bound}, we obtain
$$
d(x_i, x_{i+1}) \le d_0.
$$
$(ii)$ Let $t \geq t_0$ and let $n \in \mathbb{N}$. Let $x,y \in S_n^{+t}$. Choose $x', y' \in S_n$ such that
$$
d(x, x') \leq t \quad \text{and} \quad d(y, y') \leq t.
$$
There exist geodesic rays $c, c'$ starting from $e$, and points $x'' \in c$, $y'' \in c'$ such that
$$
d(x', x'') \leq D \quad \text{and} \quad d(y', y'') \leq D.
$$
Up to replacing $D$ by $2D$ in the previous inequalities, we may assume that $x'', y'' \in S_n$. Since $t \geq t_0$, we have
$$
d_0 \leq t \quad \text{and} \quad 2D \leq t.
$$
By considering a path in $\partial G$ from $c(+\infty)$ to $c'(+\infty)$ and applying $(i)$ to this path, we obtain a $d_0$-path in $S_n$ from $x''$ to $y''$. Replacing each step of this $d_0$-path by a geodesic segment in $G$, and since $d_0 \leq t$, we obtain a path contained in $S_n^{+t}$. Finally, by adding the geodesic segments
$$
[x, x'], \quad [x', x''], \quad [y'', y'], \quad [y', y],
$$
and since $2D \leq t$, these segments also lie in $S_n^{+t}$. Hence we obtain a path in $S_n^{+t}$ from $x$ to $y$.

$(iii)$ Let $\gamma$ be a path of length $d$ from $x$ to $y$ in $G \setminus B_{n-1}$. By Proposition~\ref{prop:basics_hyp_groups1}, $\pi_n(\gamma)$ is a $(1+2 \delta)$-path in $S_n$ from $x$ to $y$. By replacing each step of this $(1+2 \delta)$-path by a geodesic segment in $G$, and since $1+2 \delta \leq t_0 \leq t$, we obtain a path contained in $S_n^{+t}$. Moreover, its length is $\leq (1+2\delta)d$.
\end{proof}
We end this subsection by citing a result of Lazarovich--Margolis--Mj \cite[Claims 4.2 and 4.3]{lazarovich2024commensurated}, which describes minimal closed separators in the boundary and shows that they satisfy a uniform density property. This result will be crucial in the proof of \Cref{thm:cut_spheres_exp}.

\begin{thm}\label{thm:LMM}
    Let $G$ be a one-ended hyperbolic group that does not split over a two-ended subgroup, and let $\rho$ be a visual metric on $\partial G$. Then there exists a constant $\lambda \in (0,1)$ such that the following holds. If $Z \subseteq \partial G$ separates $\partial G$, and $\xi_1,\xi_2 \in \partial G \setminus Z$ belong to different connected components of $\partial G \setminus Z$, then there exists a minimal closed subset (with respect to inclusion) $Z'$ of $Z$ that separates $\xi_1$ and $\xi_2$. Moreover, for any $\xi \in Z'$ and any $0 < r < \rho \big( \{\xi_1,\xi_2\}, Z' \big)$,
    $$ Z' \cap \big( \bar B_\rho(\xi,r) \setminus B_\rho(\xi,\lambda r) \big) \neq \emptyset,  $$
    where $\bar B_\rho$ and $B_\rho$ denote the closed and open balls, respectively, in $(\partial G,\rho)$.
\end{thm}
\subsubsection{Uniform distortion of spheres in one-ended hyperbolic groups}

By Proposition~\ref{prop:basic_hyp_groups2}, for every $t \ge t_0$ and every $n \in \mathbb{N}$, $S_n^{+t}$ is connected. For $x,y \in S_n^{+t}$, we define
$d_{S_n^{+t}}(x,y)$ to be the length of a shortest path in $S_n^{+t}$ joining $x$ to $y$.

The aim of this subsection is to show that the family $(S_n^{+t},d_{S_n^{+t}})_{n \in \mathbb{N}}$ is uniformly exponentially distorted, independently of $n$:
\begin{prop}\label{prop:distortion_of_spheres}
    Let $G$ be a one-ended hyperbolic group, and let $t_0$ be as in Proposition~\ref{prop:basic_hyp_groups2}. Then for every $t \geq t_0$, there exist constants $\mu_1,\mu_2,A_1,A_2 >0$ such that the following holds. For every $n >t$, and every $x,y \in S_n^{+t}$ such that $d(x,y) \geq 9 \delta + 6t$, 
    $$ A_1 e^{\mu_1 d(x,y)} -2t \leq  d_{S_n^{+t}}(x,y) \leq A_2 e^{\mu_2 d(x,y)}.$$
\end{prop}
The proof will follow from the next lemmas. The first one is standard.
\begin{lemma}\label{lem:disto_sphere_lowerbound}
    Let $X$ be a $\delta$-hyperbolic graph, where $\delta>0$, and let $t \geq 0$. Then for every $R>t$, and every $o,x,y \in X$ such that $x,y \in S(o,R)$, the following holds. If $d(x,y) \geq 9 \delta +4t$, then any path from $x$ to $y$ in the complement of $B(o,R-t)$ has length $\geq 2^{\frac{d(x,y)}{4 \delta}-1}$.
\end{lemma}
\begin{proof}
    Set $d:= d(x,y)$, let $\alpha$ be a geodesic segment $[x,y]$ and let $m$ be its midpoint. By considering the geodesic triangle $[o,x] \cup \alpha \cup [y,o]$ and using its $\delta$-slimness, we show, as in the proof of Proposition~\ref{prop:dist_boundary_sphere_comparison}, that
    $$ d(o,m) \leq R - \frac{d}{2} + 2 \delta.$$
    By the triangle inequality, and since $d > 8 \delta +4t$, we obtain 
    $$B\left(m, \frac{d}{4}\right) \subseteq B(o,R-t).$$ 
    Let $\gamma$ be any path from $x$ to $y$ contained in $X \setminus B(o,R-t)$. Then $\gamma$ avoids $B(m,\frac{d}{4})$, and the conclusion follows from \cite[Chap.~3, Prop.~1.6]{bridson2013metric}.
\end{proof}
For the other inequality, we need one-endedness. Before that, let us recall the following. Let $D$ be the constant of Proposition~\ref{prop:basics_hyp_groups1}. In \cite{bestvina1991boundary}, Bestvina--Mess introduced the following property $({}^\ddagger_M)$, for $M>0$. We say $G$ satisfies the property $({}^\ddagger_M)$ if there exists a constant $L>0$, such that for every $R>0$ and for every pair of points $x,y \in S_R $ with $ d(x,y) \leq M$, there exists a path in $G \setminus B_{R-D}$ from $x$ to $y$ of length at most $L$. Bestvina--Mess show that if $({}^\ddagger_M)$ fails to hold for some $M>0$, then the boundary must contain a cut point. However, it was shown later by Bowditch, Levitt, and Swarup in \cite{Bowditch1999,Levitt1998,Swarup1996} that the boundary of a one-ended hyperbolic group can never contain a cut point. Therefore, $({}^\ddagger_M)$ holds for every $M>0$ for one-ended hyperbolic groups. As a consequence, we have
\begin{lemma}\label{lem:HruskRuaneLemma}(\cite[Lemma 4.1]{HruskaRuane})
Let $G$ be a one-ended hyperbolic group. There exists $\lambda>0$ such that the following holds. If $c$ and $c'$ are geodesic rays emanating from $e$, and $\alpha$ is a path of length at most $\ell$ from $c$ to $c'$ in the complement of the ball $B_r$, then there exists a path $\beta$ of length at most $\lambda \ell$ from $c$ to $c'$ in the complement of the ball $B_{r+1}$.
\end{lemma}

One can therefore ``push'' outward a path between points in $S_n$, which will yield the upper bound.

\begin{lemma}\label{lem:disto_sphere_upperbound}
    Let $G$ be a one-ended hyperbolic group, and let $t_0$ be as in Proposition~\ref{prop:basic_hyp_groups2}. There exist constants $\mu,C>0$ such that the following holds. For every $n \in \mathbb{N}$, for every $t \geq t_0$, and every $x,y \in S_n$, there exists a path from $x$ to $y$, in $S_n^{+t}$, of length $\leq C e^{\mu d(x,y)}$.
\end{lemma}
\begin{proof}
    Let $D$ be the constant of Proposition~\ref{prop:basics_hyp_groups1}, and let $L>0$ for which $({}^\ddagger_{2D})$ holds. If $d(x,y) \leq 2D$, then property $({}^\ddagger_{2D})$ implies that there exists a path $\gamma$ of length $\leq L$ from $x$ to $y$ in the complement of $B_{n-D}$. By projecting back to $S_n$ any sub-path of $\gamma$ that leaves $S_n^{+t}$, by item $(iii)$ of Proposition~\ref{prop:basic_hyp_groups2}, there is a path from $x$ to $y$ in $S_n^{+t}$ of length $\leq L (1+2\delta)$.
    
    Suppose that $d(x,y) > 2D$, and let $\gamma$ be a geodesic segment $[x,y]$. Set $d:= d(x,y)$. Then $\gamma$ lies in the complement of $B(e, n - \frac{d}{2}-1)$. Let $c,c'$ be geodesic rays emanating from $e$, and points $x' \in c$ and $y' \in c'$ such that
    $$
    d(x, x') \leq D \quad \text{and} \quad d(y, y') \leq D,
    $$
    as in Proposition~\ref{prop:basics_hyp_groups1}. Up to replacing $D$ by $3D$ in the previous inequalities, we assume that $x',y' \in S_n$. Therefore, there exists a path $\gamma'$ from $x'$ to $y'$, of length $ \leq \ell := d+2L$, that lies in the complement of $B(e, n - \frac{d}{2}-1)$. In particular, $\gamma'$ is a path between $c$ and $c'$. Therefore, there exists a path of length $\leq \lambda \ell$ between $c$ to $c'$ in the complement of $B(e, n - \frac{d}{2})$, where $\lambda$ is the constant from Lemma~\ref{lem:HruskRuaneLemma}. We can suppose that $\lambda >1$. By repeating the process $\lceil \frac{d}{2}  \rceil$ times, we get a path $\alpha$ between $c$ and $c'$ of length $\leq \lambda^{\lceil \frac{d}{2}  \rceil} \ell$ in the complement of $B(e,n-1)$. By item $(iii)$ of Proposition~\ref{prop:basic_hyp_groups2}, projecting $\alpha$ to $S_n$ yields a path $\beta$ from $x'$ to $y'$ in $S_n^{+t}$ of length $\leq (1+2\delta) \lambda^{\lceil \frac{d}{2}  \rceil} \ell$. Since $d(x,x') \leq 3D$, $d(y,y') \leq 3D$, and $3D \leq t$, there exists a path from $x$ to $y$ in $S_n^{+t}$ of length $\leq (1+2\delta) \lambda^{\lceil \frac{d}{2}  \rceil} (d+2L) + 6D$. Since $\lambda>1$, there exists a constant $C \ge 1$, depending only on $\delta,\lambda,L$ and $D$ (and not on $n$ or $t$), such that for every $d \ge 0$,
\[
(1+2\delta)\,\lambda^{\lceil d/2 \rceil}(d+2L) + 4D \le C e^{\mu d},
\]
where $\mu := \log(\lambda)$. This completes the proof.
\end{proof}
\begin{proof}[Proof of Proposition~\ref{prop:distortion_of_spheres}]
    $G$ is one-ended so $\delta >0$. Let $t \geq t_0$ be fixed. Let $n>t$ and $x,y \in S_n^{+t}$ such that $d(x,y) \geq 9 \delta + 6t$. Let $ x',y' \in S_n$ be such that 
    $$ d(x, x') \leq t \quad \text{and} \quad d(y, y') \leq t.$$
    Then $d(x',y') \geq 9 \delta +4t$. By Lemma~\ref{lem:disto_sphere_lowerbound},
    $$ d_{S_n^{+t}}(x',y') \geq 2^{\frac{d(x',y')}{4 \delta}-1}.$$
    Therefore,
    $$ d_{S_n^{+t}}(x,y) \geq 2^{\frac{d(x,y)-2t}{4 \delta}-1}-2t.$$
    Set $\mu_1 := \frac{\log 2}{4 \delta} $ and $A_1:= 2^{-\frac{t}{2\delta}-1} $. For the upper bound, we have that 
    $$d_{S_n^{+t}}(x',y') \leq C e^{\mu d(x',y')}, $$
    where $\mu,C>0$ are the constants from Lemma~\ref{lem:disto_sphere_upperbound}. Therefore,
    $$d_{S_n^{+t}}(x,y) \leq d_{S_n^{+t}}(x',y') +2t  \leq C e^{\mu (d(x,y)+2)} +2t. $$
    Set $A_2:= C e^{2 \mu} +2t$ and $\mu_2:= \mu$. 
\end{proof}

\subsubsection{Shadows}
In the rest of this subsection, $\partial G$ is equipped with its visual metric $\rho$ as in \Cref{subsection:background_prelims}, and constants $C_1,C_2$ as in Proposition~\ref{prop:dist_boundary_sphere_comparison}.
\begin{definition}\label{def:shadows}
    If $G$ is a non-elementary hyperbolic group, let $t_0 \geq 0$ be as in Proposition~\ref{prop:basic_hyp_groups2}. 
    For $x \in G$, we define its \emph{shadow} $\sh(x) \subset \partial G$ to be the set of boundary points $\xi$ for which there exists a geodesic ray $[e,\xi)$ passing through $B(x,t_0)$. If $A \subset G$ is a subset, we define its shadow by
    $$
    \sh(A) := \bigcup_{a \in A} \sh(a).
    $$
\end{definition}
By the choice of $t_0$, we have $t_0 \geq D$, where $D$ is the constant from Proposition~\ref{prop:basics_hyp_groups1}. So for every $x \in G$ there exists $x' \in B(x,t_0)$ such that the geodesic segment $[e,x']$ extends to a geodesic ray $[e,\xi)$. In particular, $\sh(x) \neq \emptyset$.
\begin{lemma}\label{Lem:empty_intersect_shadows}
    Let $G$ be a non-elementary hyperbolic group, $t_0$ as in Proposition~\ref{prop:basic_hyp_groups2}, and $t \geq t_0$. Then for any $n\in \mathbb{N}$, if $x,y \in S_n^{+t}$ and $d(x,y) \geq 5t_0+2t$, then
        $$\sh(x) \cap \sh(y) = \emptyset. $$
\end{lemma}
\begin{proof}
Let $\xi \in \sh(x)$, and let $x' \in B(x,t_0)$ be such that $[e,\xi)$ extends $[e,x']$. Let $x''$ be the intersection of $[e,\xi)$ with $S_n$, so $d(x,x'') \leq 2t_0+t$. Let $\eta \in \sh(y)$, and let $y'' \in S_n$ be constructed similarly. So $d(x'',y'') \geq 5t_0+2t - 2(2t_0+t)= t_0$. Then 
$$\rho(\xi, \eta ) \geq \frac{1}{C_1}a^{-n}>0,$$ 
where $C_1$ is the constant from Proposition~\ref{prop:dist_boundary_sphere_comparison}.
\end{proof}
A family of points $\{x_i\}_{i \in I}$ in a metric space $(X,d)$ is called \emph{$r$-separated}, for some $r\geq 0$, if
$$
d(x_i,x_j) \geq r
$$
for every distinct $i,j \in I$.
\begin{lemma}\label{lem:separated_family_boundary}
    Let $G$ be a non-elementary hyperbolic group, and let $t_0$ be as in Proposition~\ref{prop:basic_hyp_groups2}. Let $t \geq t_0, $ $n \in \mathbb{N}$, $\omega>0$, and $Z \subset S_n^{+t}$ a subset. If there exists $\{\zeta_i\}_{i \in I}$ in $\sh(Z) \subset \partial G$ a $\omega$-separated family, then there exist $\{z_i\}_{i \in I}$ in $Z$ a $d_\omega$-separated family, with
    $$ d_\omega := 2n + 2 \log_a \left(\frac{\omega}{C_2}\right) - 4t_0-2t,$$
    where $C_2$ is the constant from Proposition~\ref{prop:dist_boundary_sphere_comparison}
\end{lemma}
\begin{proof}
    For every $i \in I$, $\zeta_i \in \sh(Z)$, so there exists $z_i \in Z$ and $z_i' \in B(z_i,t_0)$ such that $[e,z_i']$ extends to $[e,\zeta_i)$. Up taking the ball $ B(z_i,2t_0+t)$ instead of $ B(z_i,t_0)$, we can assume that $z_i' \in S_n$. For every $i \ne j$ in $I$, $\rho(\zeta_i, \zeta_j) \geq \omega$, so by \eqref{eq:if_rho_big_d_big},
    $$ d(z_i',z_j') \geq 2n + 2 \log_a \left(\frac{\omega}{C_2}\right).$$
    Since $d(z_i,z_i') \leq 2t_0+t$, we get
    $$d(z_i,z_j) \geq 2n + 2 \log_a \left(\frac{\omega}{C_2}\right) -4t_0-2t.$$
    Set $d_\omega := 2n + 2 \log_a \left(\frac{\omega}{C_2}\right) -4t_0 - 2t $.
\end{proof}

\begin{prop}\label{prop:shadows_separated_by_shadows}
    Let $G$ be a one-ended hyperbolic group. Let $t_0$ be as in Proposition~\ref{prop:basic_hyp_groups2}, and let $t \geq t_0$. Then there exists a constant $\theta >0$ such that the following holds. For every $n \in \mathbb{N}$, every subset $Z \subseteq S_n^{+t}$, and every pair of points $x,y \in S_n^{+t}$ contained in different connected components of $S_n^{+t} \setminus Z$ and satisfying
    $$
    d(\{x,y\},Z) \geq \theta,
    $$
    the shadows $\sh(x)$ and $\sh(y)$ are disjoint and lie in different connected components of $\partial G \setminus \sh(Z)$. In particular, $\sh(Z)$ separates $\partial G$.
\end{prop}

\begin{proof}
    Let $A_1,A_2,\mu_1,\mu_2>0$ be the constants as in Proposition~\ref{prop:distortion_of_spheres}. Let $\theta \geq 10t$ be large enough such that 
    \begin{equation}\label{eq:def_of_theta}
         2(A_1 e^{\mu_1 \theta} -2t) >  A_2 e^{10t \mu_2 }.
    \end{equation}
    We show that the statement holds for this $\theta$. If $d(x,Z) \geq \theta$, then $d(x,Z) \geq 10t$, and $\sh(x)$ and $\sh(Z)$ are disjoint by Lemma~\ref{Lem:empty_intersect_shadows}. The same holds for $\sh(y)$ and $\sh(Z)$. Moreover, $d(x,Z) \geq \theta$ implies that
    $$ d_{S_n^{+t}}(x,Z) \geq A_1 e^{\mu_1 d(x,Z)} -2t \geq A_1 e^{\mu_1 \theta} -2t. $$
    The same holds for $y$. Since $x$ and $y$ are in different connected components of $S_n^{+t} \setminus Z$, we obtain
    $$  d_{S_n^{+t}}(x,y) \geq d_{S_n^{+t}}(x,Z)+ d_{S_n^{+t}}(y,Z) \geq 2(A_1 e^{\mu_1 \theta} -2t).  $$
    Therefore, by Proposition~\ref{prop:distortion_of_spheres},
    $$d(x,y) \geq \frac{\log\left(\frac{d_{S_n^{+t}}(x,y)}{A_2}\right)}{\mu_2} \geq 7t,$$
    where the right inequality holds by \eqref{eq:def_of_theta}. Therefore, $\sh(x)$ and $\sh(y)$ are disjoint by Lemma~\ref{Lem:empty_intersect_shadows}. 
    
    Now suppose that there exists a path $\tilde{\gamma}$ in $\partial G$ from $\sh(x)$ to $\sh(y)$ which avoids $\sh(Z)$. Let $x'$ (resp.\ $y'$) be the intersection of $S_n$ with a geodesic ray from $e$ to $\sh(x)$ (resp.\ $\sh(y)$), and let $\gamma$ be the pull-back of $\tilde{\gamma}$ to $S_n$ given by Proposition~\ref{prop:basic_hyp_groups2}. Then $\gamma$ is a $d_0$-path in $S_n$ from $x'$ to $y'$ which avoids $Z^{+t_0}$, where $d_0$ is the constant from Proposition~\ref{prop:dist_boundary_sphere_comparison}. However, since $d(x,x') \le 2t_0+t < \theta$, the points $x$ and $x'$ lie in the same connected component of $S_n^{+t} \setminus Z$, and similarly $y$ and $y'$ lie in a same connected component. Hence $x'$ and $y'$ lie in two different connected components of $S_n^{+t} \setminus Z$. In particular, every path in $S_n^{+t}$ from $x'$ to $y'$ intersects $Z$ and every $d_0$-path from $x'$ to $y'$ must intersect $Z^{+d_0}$, and in particular must intersect $Z^{+t_0}$ since $t_0 \geq d_0$. This contradicts the fact that $\gamma$ avoids $Z^{+t_0}$.
\end{proof}

\section{Proof of the main theorem}

\noindent
We are now ready to prove \Cref{thm:cut_spheres_exp}.

\begin{proof}[Proof of \Cref{thm:cut_spheres_exp}]
    Let $t_0\geq 1$ be as in Proposition~\ref{prop:basic_hyp_groups2}, and let $t \geq t_0$ and $\varepsilon \in (0,1)$. Let $n \in \mathbb{N}$, and let $Z_n \subset S_n^{+t}$ be a $\varepsilon$-cut. Let $\Theta$ and $\alpha$ be the constants from Proposition~\ref{prop:big_spheres_farfrom_cuts}. 

    \medskip\noindent
    If the case $(1)$ from Proposition~\ref{prop:big_spheres_farfrom_cuts} does not hold, then
    \begin{equation}\label{eq:size_Z_case1}
        |Z_n| \geq \Theta e^{\frac{\alpha}{2}n}.
    \end{equation}
    Otherwise, let $x_n,y_n$ be in different connected components of $S_n^{+t} \setminus Z_n$ such that 
    $$d(\{x_n,y_n\},Z_n) > \frac{n}{2}.$$ 
    So, for $n$ large enough, by Proposition~\ref{prop:shadows_separated_by_shadows}, $\sh(x_n)$ and $\sh(y_n)$ are contained in two different connected components of $\partial G \setminus \sh(Z_n)$. 

    \medskip\noindent
    Let $\xi_n \in \sh(x_n)$, $\eta_n \in \sh(y_n)$, and $\zeta \in \sh(Z_n)$. So there exits $z \in Z_n$, and $z' \in B(z,t_0)$ such that $[e,z']$ extends to $[e,\zeta)$. Up taking the ball $ B(z,3t)$ instead of $ B(z,t_0)$, we can assume that $z' \in S_n$. Similarly, let $x_n'$ be in $B(x_n,3t) \cap S_n$ such that $[e,x_n']$ extends to $[e,\xi_n)$. Since $d(x_n,Z_n) > \frac{n}{2}$, 
    $$\frac{n}{2} \leq d(x_n,z) \leq d(x_n,x_n')+d(x_n',z')+d(z',z) \leq d(x_n',z') +6t. $$
    So $d(x_n',z') \geq \frac{n}{2} - 6t$, which is greater than $d_0$ (of Proposition~\ref{prop:dist_boundary_sphere_comparison}) for $n$ large enough. Therefore, by Proposition~\ref{prop:dist_boundary_sphere_comparison}, 
    $$\rho(\xi_n, \zeta) \geq C_1 \, a^{-n} \, a^{\frac{n}{2}-6t} = \frac{C_1}{a^{6t}} a^{-\frac{n}{2}}.$$
    The same holds for $\rho(\eta_n,\zeta)$. Hence
    $$ \rho(\{\xi_n,\eta_n\},\sh(Z_n)) \geq \frac{C_1}{a^{6t}} a^{-\frac{n}{2}}. $$
    Set $r_0 := \frac{C_1}{a^{6t}} a^{-\frac{n}{2}}$. By \cite[Claims 4.2 and 4.3]{lazarovich2024commensurated}, see \Cref{thm:LMM}, there exists a minimal closed subset, with respect to inclusion, $Y \subseteq \sh(Z_n)$ that separates $\xi_n,\eta_n$. In particular
    $$\rho(\{\xi_n,\eta_n\}, Y) \geq r_0.$$
    Moreover, there exists a constant $\lambda \in (0,1)$, that only depends on $(\partial G, \rho)$, such that for all $\zeta \in Y$ and $r\in (0,r_0)$, we have
    \begin{equation}\label{eq:intersection_separatingset_boundary}
        Y \cap \left(  \bar B_\rho(\zeta,r) \setminus B_\rho(\zeta,\lambda r)   \right) \ne \emptyset,
    \end{equation}
    where $ \bar B_\rho$ denotes the closed ball in $(\partial G,\rho)$, and $B_\rho$ denotes the open ball.

    \medskip\noindent
    Since $\lambda <1$, up to replacing it by $\lambda^N$ for some $N \in \mathbb{N}$, we may assume that $\lambda< \frac{1}{3}$. Choose some $\zeta_1 \in Y$, and fix $r = \frac{r_0}{2}$.
    \begin{claim}
        For every $p \in \mathbb{N}^*$, there exist $2^p$ points in $Y$ that are $\lambda^{2p}r$-separated.
    \end{claim}
    \begin{proof}[Proof of the claim]
        We prove it by induction on $p$. For $p=1$, by \eqref{eq:intersection_separatingset_boundary}, there exists 
    $$ \zeta_2 \in Y \cap \left(  \bar B_\rho(\zeta_1,\lambda r) \setminus B_\rho(\zeta_1,\lambda^2 r)   \right) .$$ 
    So $\{\zeta_1,\zeta_2\}$ are $\lambda^2 r$-separated. 
    
    Suppose that we found, $2^p$ points $ \{\zeta_i \}_{ 1 \leq i \leq 2^p}$ in $Y$ that are $\lambda^{2p}r$-separated.  Note that, since $\lambda<\frac{1}{3}$, for any $i \ne j$ in $\{1,\dots,2^p\}$, 
    $$B_\rho(\zeta_i,\lambda^{2p+1}r) \cap  B_\rho(\zeta_j,\lambda^{2p+1}r) = \emptyset $$ 
    because $2 \lambda^{2p+1}r < \lambda^{2p}r \leq \rho(\zeta_i,\zeta_j)$. By applying \eqref{eq:intersection_separatingset_boundary}, for every $i \in \{1,\dots,2^p\}$, there exists
    $$ \zeta_i' \in Y \cap \left(  \bar B_\rho(\zeta_i,\lambda^{2p+1} r) \setminus B_\rho(\zeta_i,\lambda^{2p+2} r)   \right) .$$ 
    For any $i \ne j$, $B_\rho(\zeta_i,\lambda^{2p+1}r) \cap  B_\rho(\zeta_j,\lambda^{2p+1}r) = \emptyset $, so $$\rho(\zeta_i',\zeta_j) \geq \lambda^{2p+1}r > \lambda^{2p+2}r.$$
    Moreover,
    \begin{align*}
        \lambda^{2p}r \leq \rho(\zeta_i,\zeta_j) &\leq \rho(\zeta_i,\zeta_i') + \rho(\zeta_i',\zeta_j') + \rho(\zeta_j',\zeta_j)
        \\& \leq \lambda^{2p+1}r + \rho(\zeta_i',\zeta_j') + \lambda^{2p+1}r.
    \end{align*}
    Therefore,
    $$\rho(\zeta_i',\zeta_j') \geq \lambda^{2p}r  - 2 \lambda^{2p+1}r = \lambda^{2p}r(1-2\lambda) > \lambda^{2p+2}r, $$
    where the right inequality holds because $\lambda < \frac{1}{3}$.  Therefore, $ \{\zeta_i, \zeta_i' \}_{ 1 \leq i \leq 2^p}$ in $Y$ are $\lambda^{2(p+1)}r$-separated.
    \end{proof} 
    By Lemma~\ref{lem:separated_family_boundary}, every $\lambda^{2p}r$-separated family in $Y \subset \sh(Z_n)$ yields a  $d_p$-separated family in $Z_n$, where 
    $$d_p := 2n + 2 \log_a \left(\frac{\lambda^{2p}r}{C_2}\right) - 4t_0-2t.$$
    Therefore, to get distinct points in $G$, one should have $d_p \geq 1$. Since $r = \frac{r_0}{2} = \frac{C_1}{2a^{6t}} a^{-\frac{n}{2}} $,
    we have 
    $$
    \begin{aligned}
d_p \geq 1 &\iff 2n + 2 \log_a \left(\frac{\lambda^{2p} \frac{C_1}{2a^{6t}} a^{-\frac{n}{2}} }{C_2}\right) - 4t_0-2t \geq 1 \\
 &\iff p \leq \tau \, n+ \Omega,
\end{aligned}
$$
where
$$ \tau :=\frac{1}{4 \log_a\left(\frac{1}{\lambda}\right)} \qquad \textup{ and } \qquad \Omega :=  \frac{2 \,\log_a \left(\frac{C_1}{2C_2}\right) - 14t - 4t_0 - 1}{4  \log_a \left(\frac{1}{\lambda}\right)} . $$
Note that $\tau$ and $\Omega$ only depend on $G$ and $t$, and that $\tau >0$ because $\lambda <1$. Therefore $p$ can be taken $ p = \lfloor \tau \, n+ \Omega \rfloor \geq \tau \, n+ \Omega -1$. So $Z_n$ contains a family of at least $2^{\Omega-1} 2^{\tau n}$ points that are $1$-separated. Therefore
$$|Z_n| \geq 2^{\Omega-1} 2^{\tau n}.$$
Combining this with \eqref{eq:size_Z_case1}, we conclude that
$$ |Z_n| \geq \min\{ 2^{\Omega-1} 2^{\tau n}, \Theta e^{\frac{\alpha}{2}n}  \} . \qedhere $$
\end{proof}

\begin{proof}[Proof of \Cref{thm:IntroHyp}]
    If $G$ splits over a two-ended subgroup $H$, then $G$ is coarsely separated by $H$. Moreover, since $H$ is virtually cyclic, and cyclic subgroups are undistorted in hyperbolic groups, then $H$ is undistorted. Therefore $G$ is coarsely separated by a subset of linear growth.
    
    Suppose that $G$ does not split over a two-ended subgroup, and let $t \geq 0$ be as in \Cref{thm:cut_spheres_exp}. Since the family $\{S_G(g,n)^{+t}\}_{g \in G,\, n \ge t}$ is persistent (see Example~\ref{exple:persistent}), it follows from \Cref{thm:persist_family_expcut_coarse_sep} that $G$ cannot be coarsely separated by any family of subsets with sub-exponential growth.
\end{proof}

\section{Applications to graph products}

This section is dedicated to the proof of \Cref{thm:GPsubsep}. We start by recalling basic definitions and properties related to graph products in Section~\ref{section:GP}. Section~\ref{section:GPHyp} is dedicated to the hyperbolicity of such groups. In particular, we determine precisely when a graph product is a virtual surface group. In Section~\ref{section:GPsplit}, we study splittings of graph products, and in particular we characterise splittings over virtually cyclic groups. Finally, we prove Theorem~\ref{thm:GPsubsep} in Section~\ref{section:GPproof}, where we also mention a couple of concrete examples.

\subsection{Graph products}\label{section:GP}

\noindent
Given a graph $\Gamma$ and a collection of groups $\mathcal{G}=\{G_u \mid u \in V(\Gamma) \}$ indexed by the vertices of $\Gamma$, the \emph{graph product} $\Gamma \mathcal{G}$ is the group defined by the relative presentation
$$\langle G_u \ (u \in V(\Gamma)) \mid [G_u,G_v]=1 \ (\{u,v\} \in E(\Gamma)) \rangle,$$
where $E(\Gamma)$ denotes the edge-set of $\Gamma$ and where $[G_u,G_v]=1$ is a shorthand for ``$[a,b]=1$ for all $a \in G_u$ and $b \in G_v$. We refer to the groups in $\mathcal{G}$ as the \emph{vertex-groups}.

\medskip \noindent
Usually, one says that graph products interpolate between free products (when $\Gamma$ has no edge, i.e.\ ``nothing commute'') and direct sums (when $\Gamma$ is a complete graph, i.e.\ ``everything commute''). Classical examples also include right-angled Artin groups (= graph products of infinite cyclic groups) and right-angled Coxeter groups (= graph products of cyclic groups of order $2$). 

\medskip \noindent
\textbf{Convention.} In the rest of the article, we will always assume that the vertex-groups of our graph products are non-trivial. 

\medskip \noindent
Given a graph $\Gamma$ and a collection of groups $\mathcal{G}$ indexed by $V(\Gamma)$, an element $g \in \Gamma \mathcal{G}$ can be written as a product $s_1 \cdots s_n$ where each $s_i$ belongs to a vertex-group. We refer to such a product as a \emph{word} and to each $s_i$ as a \emph{syllable}. Notice that, given such a word $r_1 \cdots r_m$, applying the following operations does not modify the element of $\Gamma \mathcal{G}$ it represents:
\begin{description}
    \item[(Cancellation)] if there exists some $1 \leq i \leq m$ such that $r_i=1$, remove the syllable $r_i$;
    \item[(Merging)] if there exists some $1 \leq i \leq m-1$ such that $r_i$ and $r_{i+1}$ belong to the same vertex-group, replace the subword $r_ir_{i+1}$ with the single syllable $(r_ir_{i+1})$;
    \item[(Shuffling)] if there exists some $1 \leq i \leq m-1$ such that $r_i$ and $r_{i+1}$ belong to adjacent vertex-groups, replace the subword $r_ir_{i+1}$ with the subword $r_{i+1}r_i$.
\end{description}
A word is \emph{graphically reduced} if it cannot be shortened by applying a sequence of such moves. 

\begin{prop}[\cite{GreenGP}]\label{prop:NormalForm}
Let $\Gamma$ be a graph and $\mathcal{G}$ a collection of groups indexed by $V(\Gamma)$. Every element of $\Gamma \mathcal{G}$ can be represented by a graphically reduced word. Moreover, any two such words only differ from each other by a sequence of shufflings.
\end{prop}

\noindent
In the following, we will need some vocabulary related to graphs, which we introduce now. 
\begin{itemize}
    \item A \emph{join} $A \ast B$ of two graphs $A$ and $B$ is the graph obtained from the disjoint union $A \sqcup B$ by connecting with an edge every vertex of $A$ to every vertex of $B$. 
    \item Given a graph $\Gamma$ and a vertex $u \in V(\Gamma)$, the \emph{link} of $u$ in $\Gamma$, denoted $\mathrm{link}(u)$, is the subgraph of $\Gamma$ induced by the neighbours of $u$. 
    \item Given a graph $\Gamma$ and a vertex $u \in V(\Gamma)$, the \emph{star} of $u$ in $\Gamma$, denoted $\mathrm{star}(u)$, is the subgraph of $\Gamma$ induced by $\mathrm{link}(u) \cup \{u\}$.
\end{itemize}
Finally, given a graph $\Gamma$, a collection of groups $\mathcal{G}$ indexed by $V(\Gamma)$, and a subgraph $\Lambda \leq \Gamma$, we denote by $\langle \Lambda \rangle$ the subgroup of $\Gamma \mathcal{G}$ generated by the vertex-groups indexed by the vertices of $\Lambda$. Notice that, if $\Lambda$ decomposes as a join $\Phi \ast \Psi$, then $\langle \Lambda \rangle = \langle \Phi \rangle \oplus \langle \Psi \rangle$.

\subsection{Hyperbolicity}\label{section:GPHyp}

\noindent
In this section, we focus on the hyperbolicity of graph products. First of all, we record the following characterisation:

\begin{thm}[\cite{GPHyp}]\label{thm:GPhyp}
Let $\Gamma$ be a finite graph and $\mathcal{G}$ a collection of finitely generated groups indexed by $V(\Gamma)$. The graph product $\Gamma \mathcal{G}$ is hyperbolic if and only if the following conditions hold:
\begin{itemize}
    \item every group of $\mathcal{G}$ is hyperbolic;
    \item no two infinite vertex-groups are adjacent in $\Gamma$;
    \item two vertex-groups adjacent to a common infinite vertex-group must be adjacent;
    \item the graph $\Gamma$ is $\square$-free.
\end{itemize}
\end{thm}

\noindent
Now, we would like to identify when a given graph product belong to some specific families of hyperbolic groups.

\paragraph{When virtually cyclic.} We first consider the family of virtually cyclic groups. 

\begin{prop}\label{prop:GPvirtCyclic}
Let $\Gamma$ be a finite graph and $\mathcal{G}$ a collection of groups indexed by $V(\Gamma)$. The graph product $\Gamma \mathcal{G}$ is virtually cyclic if and only if 
\begin{itemize}
	\item either $\Gamma$ is a complete graph all of whose vertices are labelled by finite groups;
	\item or $\Gamma$ is a complete graph with one vertex labelled by a virtually infinite cyclic group and all the other vertices labelled by finite groups;
	\item or $\Gamma$ is a join between a complete graph all of whose vertices are labelled by finite groups and two non-adjacent vertices labelled by $\mathbb{Z}_2$.
\end{itemize}
\end{prop}

\noindent
Our proof of the proposition will be based on the following observation:

\begin{lemma}\label{lem:GPwhenFI}
Let $\Gamma$ be a graph, $\mathcal{G}$ a collection a group indexed by $V(\Gamma)$, and $\Lambda \leq \Gamma$ a subgraph. The subgroup $\langle \Lambda \rangle$ has finite index in $\Gamma \mathcal{G}$ if and only if $\Gamma= \mathrm{star}(\Lambda)$ and $\mathrm{link}(\Lambda)$ is a finite complete graph all of whose vertices are labelled by finite groups.
\end{lemma}

\begin{proof}
If $\Gamma= \mathrm{star}(\Lambda)$ and $\mathrm{link}(\Lambda)$ is a finite complete graph all of whose vertices are labelled by finite groups, then $\Gamma \mathcal{G}= \langle \Lambda \rangle \oplus \langle \mathrm{link}(\Lambda) \rangle$ where $\langle \mathrm{link}(\Lambda) \rangle$ is finite (as a product of finitely many finite groups). Consequently, $\langle \Lambda \rangle$ has finite index in $\Gamma \mathcal{G}$. 

\medskip \noindent
Conversely, assume that $\langle \Lambda \rangle$ has finite index in $\Gamma \mathcal{G}$. If $\Gamma \neq \mathrm{star}(\Lambda)$, then we can find two vertices $u \notin V(\mathrm{link}(\Lambda) )$ and $v \in V(\Lambda)$ that are not adjacent. Fix two non-trivial elements $a \in \langle u\rangle$ and $b \in \langle v \rangle$. By noticing that $(ab)^n$ does not belong to $\langle \Lambda \rangle$ for every $n \geq 1$, we deduce that $\langle \Lambda \rangle$ cannot have finite index in $\Gamma \mathcal{G}$. Consequently, we must have $\Gamma= \mathrm{star}(\Lambda)$. This implies that $\Gamma \mathcal{G}= \langle \Lambda \rangle \oplus \langle \mathrm{link}(\Lambda) \rangle$. Clearly, $\langle \Lambda \rangle$ has finite index in $\Gamma \mathcal{G}$ if and only if $\langle \mathrm{link}(\Lambda) \rangle$ is finite, which holds if and only if $\mathrm{link}(\Lambda)$ is a finite complete graph all of whose vertices are labelled by finite groups. 
\end{proof}

\begin{proof}[Proof of Proposition~\ref{prop:GPvirtCyclic}.]
Decompose $\Gamma$ as a join $\Gamma_0 \ast \Gamma_1 \ast \cdots \ast \Gamma_n$ where $\Gamma_0$ is a complete graph (possibly empty) and where $\Gamma_1, \ldots, \Gamma_n$ are not joins and each contains at least two vertices. Notice that
$$\Gamma \mathcal{G}= \langle \Gamma_0 \rangle \oplus \langle \Gamma_1 \rangle \oplus \cdots \oplus \langle \Gamma_n \rangle;$$
and that, for every $1 \leq i \leq n$, $\langle \Gamma_i \rangle$ is infinite since $\Gamma_i$ is not complete. Thus, if $n \geq 2$, then $\Gamma \mathcal{G}$ cannot be virtually cyclic since it would contain a product of two infinite groups. Similarly, if $n =1$ and $\Gamma_0$ has a vertex labelled by an infinite group, then $\Gamma \mathcal{G}$ cannot be virtually cyclic. So only two cases remain to be considered. 

\medskip \noindent
The first case is $n=0$. In this case, $\Gamma$ is a complete graph, which amounts to saying that $\Gamma \mathcal{G}$ is the direct sum of its vertex-groups. Thus, $\Gamma \mathcal{G}$ is virtually cyclic if and only if either vertex-groups are all finite (i.e.\ $\Gamma$ is a complete graph all of whose vertices are labelled by finite groups) or there is exactly one vertex-group that is virtually infinite cyclic while all the others are finite (i.e.\ $\Gamma$ is a complete graph with one vertex labelled by a virtually infinite cyclic group and all the other vertices labelled by finite groups). 

\medskip \noindent
The second case is $n=1$ and $\Gamma_0$ is a complete graph all of whose vertices are labelled by finite groups. Then, $\Gamma \mathcal{G}= \langle \Gamma_0 \rangle \oplus \langle \Gamma_1 \rangle$ with $\langle \Gamma_0 \rangle$ finite, so $\Gamma \mathcal{G}$ is virtually cyclic if and only if $\langle \Gamma_1 \rangle$ is virtually cyclic. Because $\Gamma_1$ is not complete and contains at least two vertices, it must contain two non-adjacent vertices $u,v \in V(\Gamma_1)$. Since $\langle u ,v \rangle \simeq \langle u \rangle \ast \langle v \rangle$ is necessarily virtually cyclic, the groups labelling $u$ and $v$ must be both $\mathbb{Z}_2$. Then, $\langle u, v \rangle$ is an infinite dihedral group; and, since it must have finite index in $\langle \Gamma_1 \rangle$, we deduce from Lemma~\ref{lem:GPwhenFI} and from the fact that $\Gamma_1$ is not a join, that $\Gamma_1$ is reduced to the pair of non-adjacent vertices $\{u,v\}$. Thus, we have proved that $\Gamma$ is a join between a finite complete graph all of whose vertices are labelled by finite groups (namely, $\Gamma_0$) and two non-adjacent vertices labelled by $\mathbb{Z}_2$ (namely, $u$ and $v$). 
\end{proof}

\paragraph{When a surface group.} Then, we would like to determine when a graph product is virtually a surface group. Here, we refer to a surface group as the fundamental group of a closed surface of genus $\geq 2$. 

\begin{prop}\label{prop:NotVirtSurface}
Let $\Gamma$ be a finite graph and $\mathcal{G}$ a collection of groups indexed by $V(\Gamma)$. The graph product $\Gamma \mathcal{G}$ is virtually a surface group if and only if 
\begin{itemize}
	\item either $\Gamma$ is a complete graph with one vertex labelled by a virtual surface group and all the other vertices by finite groups; 
	\item or $\Gamma$ decomposes as a join between a complete graph all of whose vertices are labelled by finite groups and a cycle of length $\geq 5$ all of whose vertices are labelled by $\mathbb{Z}_2$.
\end{itemize}
\end{prop}

\noindent
Before turning to the proof of our characterisation, we observe that:

\begin{prop}\label{prop:BiLipEmbed}
Let $\Gamma$ be a finite graph and $\mathcal{G},\mathcal{H}$ two collections of finitely generated groups indexed by $V(\Gamma)$. If there exists a biLipschitz embedding $\iota_u : G_u \hookrightarrow H_u$ for every $u \in V(\Gamma)$, then there exists a biLipschitz embedding $\Gamma \mathcal{G} \hookrightarrow \Gamma \mathcal{H}$. 
\end{prop}

\noindent
The proof of our proposition will be based on the following technical lemma:

\begin{lemma}\label{lem:NFmedian}
Let $\Gamma$ be a graph and $\mathcal{G}$ a collection of groups indexed by $V(\Gamma)$. Any two elements $x,y \in \Gamma \mathcal{G}$ can be represented respectively as graphically reduced words $p a_1 \cdots a_n r$ and $p b_1 \cdots b_n s$ such that $s^{-1}(a_1b_1^{-1}) \cdots (a_nb_n^{-1}) r$ is a graphically reduced word representing $y^{-1}x$. 
\end{lemma}

\begin{proof}
We prove our lemma by using the geometric perspective given by \cite{QM}. We know from \cite[Proposition~8.2]{QM} that the Cayley graph
$$\mathrm{QM}(\Gamma, \mathcal{G}):= \mathrm{Cayl} \left( \Gamma \mathcal{G}, \bigcup\limits_{G \in \mathcal{G}} G \right)$$
is quasi-median, so it follows from \cite[Proposition~2.84]{QM} that there exist an equilateral triangle $p,pq_x,pq_y \in \Gamma \mathcal{G}$ such that equalities
$$\left\{ \begin{array}{l} d(1,x)=d(1,p)+d(p,pq_x) + d(pq_x,x) \\ d(1,y) = d(1,p)+d(p,pq_y)+d(pq_y,y) \\ d(x,y) = d(x,pq_x)+d(pq_x,pq_y)+d(pq_y,y) \end{array} \right.$$
hold in $\mathrm{QM}(\Gamma, \mathcal{G})$. Since geodesics in $\mathrm{QM}(\Gamma, \mathcal{G})$ are given by graphically reduced words \cite[Lemma~8.3]{QM}, the first (resp.\ second) equality shows that $x$ (resp.\ $y$) can be written as a graphically reduced word $pq_x r$ (resp. $pq_ys$). Then, the third equality shows that $y^{-1}x$ can be written as a graphically reduced word $s^{-1}qr$ where $q$ is a graphically reduced word representing $q_y^{-1}q_x$. But we know from \cite[Proposition~8.2 and Corollary~8.7]{QM} that $q_x$ and $q_y$ belong to $\langle \Lambda \rangle$ for some complete subgraph $\Lambda \leq \Gamma$, so we can write $q_x$ and $q_y$ respectively as graphically reduced words $a_1 \cdots a_n$ and $b_1 \cdots b_m$ of pairwise commuting syllables. Notice that, since our triangle $p,pq_x,pq_y$ is equilateral, necessarily $n=m$ and $(b_1^{-1}a_1) \cdots (b_n^{-1}a_n)$ is graphically reduced. The desired conclusion follows.  
\end{proof}

\begin{proof}[Proof of Proposition~\ref{prop:BiLipEmbed}.]
For convenience, let $\iota : \bigcup \mathcal{G} \hookrightarrow \bigcup \mathcal{H}$ denote the map that restricts to $\iota_u$ on each $G_u$. Given an element $g \in \Gamma \mathcal{F}$, we write $g$ as a graphically reduced word $s_1 \cdots s_n$ and we define $\Phi(g):= \iota(s_1) \cdots \iota(s_n)$. Notice that, as a consequence of Proposition~\ref{prop:NormalForm}, the element of $\Gamma \mathcal{H}$ that $\Phi(g)$ defines does not depend on the choice of the graphically reduced word representing $g$. We claim that the map $\Phi : \Gamma \mathcal{G} \to \Gamma \mathcal{H}$ is a biLipschitz embedding.

\medskip \noindent
So let $x,y \in \Gamma \mathcal{G}$ be two elements. Write $x$ and $y$ respectively as graphically reduced words $pa_1 \cdots a_n r$ and $p b_1 \cdots b_n s$ as given by Lemma~\ref{lem:NFmedian}. Let $r_1 \dots r_k$ (resp.\ $s_1 \cdots s_\ell$) be a graphically reduced word representing $r$ (resp.\ $s$). On the one hand, 
$$d(x,y) = \sum\limits_{i=1}^\ell \|s_i\| + \sum\limits_{i=1}^n \| b_i^{-1}a_i\| + \sum\limits_{i=1}^k \|r_i\|,$$
where $\|\cdot\|$ denotes the word-length of the syllable under consideration inside the vertex-group that contains it. On the other hand, 
$$d(\Phi(x),\Phi(y)) = \sum\limits_{i=1}^\ell \|\iota(s_i)\| + \sum\limits_{i=1}^n \| \iota(b_i)^{-1}\iota(a_i)\| + \sum\limits_{i=1}^k \|\iota(r_i)\|,$$
hence
$$d(\Phi(x),\Phi(y)) \geq \frac{1}{L} \left( \sum\limits_{i=1}^\ell \|s_i\| + \sum\limits_{i=1}^n \| b_i^{-1}a_i\| + \sum\limits_{i=1}^k \|r_i\| \right) = \frac{1}{L} d(x,y)$$
and
$$d(\Phi(x), \Phi(y)) \leq L \left( \sum\limits_{i=1}^\ell \|s_i\| + \sum\limits_{i=1}^n \| b_i^{-1}a_i\| + \sum\limits_{i=1}^k \|r_i\| \right)=L d(x,y),$$
where $L>0$ is a constant such that each $\iota_u$ is $L$-biLipschitz. 
\end{proof}

\begin{proof}[Proof of Proposition~\ref{prop:NotVirtSurface}.]
Decompose $\Gamma$ as a join $\Gamma_0 \ast \Gamma_1 \ast \cdots \ast \Gamma_n$ where $\Gamma_0$ is a complete graph and where $\Gamma_1, \ldots, \Gamma_n$ are not joins and each contains at least two vertices. Notice that, for every $1 \leq i \leq n$, $\langle \Gamma_i \rangle$ is infinite since $\Gamma_i$ is not complete and contains at least two vertices. Therefore, if $n \geq 2$, then $\Gamma \mathcal{G}$ contains a product of two infinite groups and consequently cannot be hyperbolic (and a fortiori not a virtual surface group). If $n=0$, then $\Gamma= \Gamma_0$ is complete and $\Gamma \mathcal{G}$ coincides with the product of its vertex-groups. In order to be a virtual surface group, again because a virtual surface group (and more generally any hyperbolic group) does not contain a product of two infinite groups, exactly one vertex of $\Gamma$ must be labelled by a virtual surface group and all the other vertices must be labelled by finite groups. It remains to consider the case $n=1$. Since $\Gamma \mathcal{G} = \langle \Gamma_0 \rangle \oplus \langle \Gamma_1 \rangle$ with $\langle \Gamma_1 \rangle$ infinite, necessarily $\langle \Gamma_0 \rangle$ must be finite, i.e.\ all the vertices of $\Gamma_0$ must be labelled by finite groups. Notice that, then, $\Gamma \mathcal{G}$ is a virtual surface group if and only if $\langle \Gamma_1 \rangle$ is virtually a surface group.

\medskip \noindent
Thus, in order to conclude the proof of our proposition, it remains to verify that, if $\Gamma$ is not a join and contains at least two vertices, then $\Gamma \mathcal{G}$ is a virtual surface group if and only if $\Gamma$ is a cycle of length $\geq 5$ all of whose vertices are labelled by $\mathbb{Z}_2$. 

\medskip \noindent
If $\Gamma$ is chordal, i.e.\ it does not contain an induced cycle of length $\geq 4$, then we know from \cite{MR130190} that it contains a separating complete subgraph, say $\Lambda \leq \Gamma$. We choose $\Lambda$ with the least number of vertices. If all the vertices of $\Lambda$ are labelled by finite groups, then $\Gamma \mathcal{G}$ splits over a finite subgroup (namely, $\langle \Lambda \rangle$), which implies that $\Gamma \mathcal{G}$ is multi-ended, and consequently not a virtual surface group. So there exists $u \in V(\Lambda)$ such that $\langle u \rangle$ is infinite. If $\mathrm{link}(u)$ is not complete, i.e.\ there exist two non-adjacent vertices $v,w$ adjacent to $u$, then $\Gamma \mathcal{G}$ cannot be hyperbolic (and a fortiori not a virtual surface group) since it contains a product of two infinite groups, namely $\langle u,v,w \rangle \simeq \langle u \rangle \oplus (\langle v \rangle \ast \langle w \rangle)$. So $\mathrm{link}(u)$ must be complete. But then $\Lambda \backslash \{u\}$ must also separate $\Gamma$, contradicting the minimality of $\Lambda$. Thus, we have proved that $\Gamma$ cannot be chordal.

\medskip \noindent
In other words, $\Gamma$ must contain a cycle of length $\geq 4$. If $(r,s,t,u)$ is a cycle of length $4$ in $\Gamma$, then $\Gamma \mathcal{G}$ cannot be hyperbolic (and a fortiori not a virtual surface group) since it contains a product of two infinite groups, namely $\langle r,s,t,u \rangle = (\langle r \rangle \ast \langle t \rangle) \oplus (\langle s \rangle \ast \langle u \rangle)$.  Thus, $\Gamma$ must contain a cycle $\Xi$ of length $\geq 5$. 

\medskip \noindent
First, assume that all the vertices of $\Xi$ are labelled by $\mathbb{Z}_2$. Then, $\langle \Xi \rangle$ is a virtual surface group. Since infinite-index subgroups in surface groups are free, it follows that $\Gamma \mathcal{G}$ is a virtual surface group if and only if $\langle \Xi \rangle$ has finite index in $\Gamma \mathcal{G}$. According to Lemma~\ref{lem:GPwhenFI}, and because $\Gamma$ is not a join, this amounts to saying that $\Gamma= \Xi$. In other words, $\Gamma$ is a cycle of length $\geq 5$ all of whose vertices are labelled by $\mathbb{Z}_2$.

\medskip \noindent
Next, assume that at least one vertex of $\Xi$ is labelled by a group of size $\geq 3$. As a consequence of Proposition~\ref{prop:BiLipEmbed}, there exists a quasi-isometric embedding $\Xi \mathcal{H} \to \Gamma \mathcal{G}$ where $\mathcal{H}$ is a collection with one $\mathbb{Z}_3$ while all its other groups are $\mathbb{Z}_2$. In order to conclude the proof of our proposition, we need to verify that $\Gamma \mathcal{G}$ is not a virtual surface group, which is a consequence of the following observation:

\begin{claim}
The group $\Xi \mathcal{H}$ is a one-ended hyperbolic group that is not virtually a surface group.
\end{claim}

\noindent
We know from Theorem~\ref{thm:GPhyp} and \cite{GPends} (see also Corollary~\ref{cor:GPends} below) that $\Xi \mathcal{H}$ is a one-ended hyperbolic group. It follows from \cite[Lemma~21]{MR3010817} that $\Xi \mathcal{H}$ contains a subgroup isomorphic to the graph product $\Theta(2)$ of cyclic groups $\mathbb{Z}_2$ (i.e.\ a right-angled Coxeter group) where $\Theta$ is the union of three paths of length $(\mathrm{length}(\Xi)-2)$ glued along their endpoints. Two such paths yield a cycle $\Omega$ of length $2 ( \mathrm{length}(\Xi)-2) \geq 6$. Thus, $\Theta(2)$ contains a virtual surface subgroup, namely $\langle \Omega \rangle$, which has infinite index according to Lemma~\ref{lem:GPwhenFI}. Since infinite-index subgroups in surface groups are free, it follows that $\Theta(2)$, and a fortiori $\Xi \mathcal{H}$, is not a virtual surface group.
\end{proof}

\subsection{Splittings}\label{section:GPsplit}

\noindent
The last step towards the proof of Theorem~\ref{thm:GPsubsep} is to understand when a graph product of finite groups splits over a virtually cyclic group. Our characterisation will follow from the our next observation, which follows the lines of \cite{MR3728497}. Given a group $G$ and a collection of subgroups $\mathcal{H}$, we refer to a \emph{splitting of $G$ relative to $\mathcal{H}$} as a splitting of $G$ for which every subgroup in $\mathcal{H}$ is elliptic in the corresponding Bass-Serre tree (or equivalently, is contained in a conjugate of a factor of the decomposition). 

\begin{prop}\label{prop:RelativeSplitting}
Let $\Gamma$ be a finite graph, $\mathcal{G}$ a collection of groups indexed by $V(\Gamma)$, and $H \leq \Gamma \mathcal{G}$ a subgroup. If $\Gamma \mathcal{G}$ splits over $H$ relative to $\mathcal{G}$, then there exists a separating subgraph $\Lambda \leq \Gamma$ such that $\langle \Lambda \rangle$ has a conjugate contained in $H$. 
\end{prop}

\begin{proof}
We know that $\Gamma \mathcal{G}$ acts non-trivially on some tree $T$ with  edge-stabilisers conjugate to $H$ and with elliptic vertex-groups, namely the Bass-Serre tree given by our splitting. For every vertex $u \in V(\Gamma)$, consider the non-empty subtree $T_u:= \mathrm{Fix}(\langle u \rangle)$. If the $T_u$ pairwise intersect, then they must globally intersect. In other words, we find a point globally fixed by $\Gamma \mathcal{G}$, but we know that it is not the case. Consequently, there exist $u,v \in V(\Gamma)$ such that $T_u \cap T_v= \emptyset$. Fix an edge $e$ of the geodesic connecting $T_u$ and $T_v$ and let $\Lambda$ denote the subgraph of $\Gamma$ induced by the vertices $w \in V(\Gamma)$ for which $\langle w \rangle$ fixes $e$ (or equivalently, $e \in E(T_w)$). Since $T_a \cap T_b \neq \emptyset$ if $a,b \in V(\Gamma)$ are adjacent, necessarily $\Lambda$ is a non-empty subgraph that separates $u$ and $v$ in $\Gamma$. Since $\langle \Lambda \rangle$ is contained in $\mathrm{stab}(e)$, which is conjugate to $H$, the desired conclusion follows.
\end{proof}

\begin{cor}\label{cor:GPends}
Let $\Gamma$ be a finite graph and $\mathcal{G}$ a collection of finite groups indexed by $V(\Gamma)$. The graph product $\Gamma \mathcal{G}$ 
\begin{itemize}
    \item splits over a finite subgroup, i.e.\ is multi-ended, if and only if $\Gamma$ contains a separating complete subgraph;
    \item splits over a virtually cyclic subgroup if and only if it contains a separating subgraph that is a join $A \ast B$ between a complete graph $A$ and a graph $B$ that is either empty or two non-adjacent vertices labelled by $\mathbb{Z}_2$. 
\end{itemize}
\end{cor}

\begin{proof}
If $\Gamma \mathcal{G}$ splits over a finite subgroup, then we know from Proposition~\ref{prop:RelativeSplitting} that $\Gamma$ contains a separating subgraph $\Lambda$ such that $\langle \Lambda \rangle$ is finite. Necessarily, $\Lambda$ must be a complete subgraph all of whose vertices are labelled by finite groups. Conversely, it is clear that, if we know that $\Gamma$ contains such a subgraph, then $\Gamma \mathcal{G}$ splits over a finite subgroup.

\medskip \noindent
If $\Gamma \mathcal{G}$ splits over a virtually cyclic group, then we know from Proposition~\ref{prop:RelativeSplitting} that $\Gamma$ contains a separating subgraph $\Lambda$ such that $\langle \Lambda \rangle$ is virtually cyclic. According to Proposition~\ref{prop:GPvirtCyclic}, $\Lambda$ is either a complete graph all of whose vertices are labelled by finite groups or a join between a complete graph all of whose vertices are labelled by finite groups and two non-adjacent vertices labelled by $\mathbb{Z}_2$. Conversely, it is clear that, if we know that $\Gamma$ contains such a subgraph, then $\Gamma \mathcal{G}$ splits over a virtually cyclic subgroup.
\end{proof}

\subsection{Proof of Theorem~\ref{thm:GPsubsep} and examples}\label{section:GPproof}

\noindent
We are finally ready to prove our main result about hyperbolic graph products of finite groups. 

\begin{proof}[Proof of Theorem~\ref{thm:GPsubsep}.]
If $\Gamma$ contains a separating subgraph $\Lambda=\Lambda_1 \ast \Lambda_2$ where $\Lambda_1$ is a complete graph all of whose vertices are labelled by finite groups and where $\Lambda_2$ is either empty or two non-adjacent vertices both labelled by $\mathbb{Z}_2$, then $\Gamma \mathcal{G}$ splits over $\langle \Lambda \rangle$, which is virtually cyclic. This implies that $\Gamma \mathcal{G}$ is coarsely separable by a family of subexponential growth.

\medskip \noindent
Conversely, assume that $\Gamma \mathcal{G}$ is coarsely separable by a family of subexponential growth. Notice that we know from Theorem~\ref{thm:GPhyp} that $\Gamma \mathcal{G}$ is hyperbolic. Moreover, Corollary~\ref{cor:GPends} shows that $\Gamma \mathcal{G}$ is multi-ended if and only if $\Gamma$ contains a separating subgraph all of whose vertices are labelled by finite groups; and Proposition~\ref{prop:NotVirtSurface} shows that $\Gamma \mathcal{G}$ is virtually a surface group if and only if $\Gamma$ decomposes as a join between a complete graph all of whose vertices are labelled by finite groups and a cycle of length $\geq 5$ all of whose vertices are labelled by $\mathbb{Z}_2$. In the latter case, notice that $\Gamma$ contains a separating subgraph that is a join between a complete graph all of whose vertices are labelled by finite groups and two non-adjacent vertices labelled by $\mathbb{Z}_2$. 

\medskip \noindent
From now on, assume that $\Gamma \mathcal{G}$ is one-ended and is not virtually a surface group. Then, we know from Theorem~\ref{thm:IntroHyp} that $\Gamma \mathcal{G}$ splits over a virtually cyclic group. Then, the desired conclusion if provided by Corollary~\ref{cor:GPends}. 
\end{proof}

\noindent
We conclude this section with a few concrete examples. First, as an immediate consequence of Theorem~\ref{thm:GPsubsep}:

\begin{cor}
Let $\Gamma$ be a finite $\square$-free graph. The right-angled Coxeter group $C(\Gamma)$ is coarsely separable by a family of subexponential growth if and only if $\Gamma$ contains a separating subgraph $A \ast B$ where $A$ is complete and where $B$ is either complete or consists of a pair of non-adjacent vertices.
\end{cor}

\noindent
Our next examples will play an important role in our forthcoming work \cite{SepRAAGs}, dedicated to coarse separation in right-angled Artin groups. They also include the so-called Bourdon groups, hence extending our previous result \cite[Theorem~1.3]{bensaid2024coarse}. 

\begin{ex}
Let $\Gamma$ be a be a cycle of length $\geq 5$ and $\mathcal{G}$ a collection of finite groups indexed by $V(\Gamma)$. The graph product $\Gamma \mathcal{G}$ is coarsely separable by a family of subexponential growth if and only if $\Gamma$ has two non-adjacent vertices that are both labelled by~$\mathbb{Z}_2$. 
\end{ex}

\noindent
Finally, let us illustrate the fact that Theorem~\ref{thm:GPsubsep} can have concrete applications to the problem of determining whether or not there exist coarse embeddings between finitely generated groups. 

\begin{ex}\label{ex:CoarseEmb}
Let $G_1$ and $G_2$ be the two graph products illustrated by Figure~\ref{GP}. There is no coarse embedding $G_1 \to G_2$. Indeed, as indicated on the figure, $G_2$ can be decomposed as an amalgamated product over an infinite dihedral group of two smaller graph products, which turn out to be virtually free according to \cite{MR2294621}. Because $G_1$ is not coarsely separable by a family of subexponential growth, then it follows from \cite[Theorem~4.1]{bensaid2024coarse} that every coarse embedding $G_1 \to G_2$ must have its image contained in a neighbourhood of one these virtually free factors, which which would imply that $G_1$ is a quasi-tree. 
\end{ex}
\begin{figure}
\begin{center}
\includegraphics[width=0.7\linewidth]{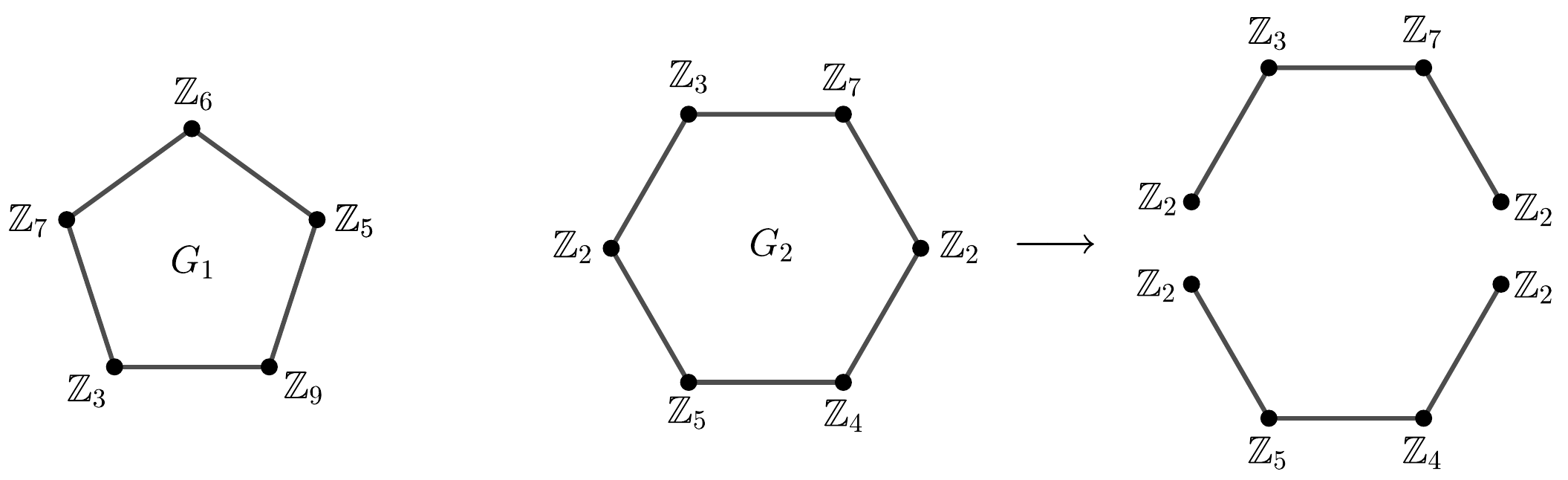}
\caption{The two graph products from Example~\ref{ex:CoarseEmb}.}
\label{GP}
\end{center}
\end{figure}

\addcontentsline{toc}{section}{References}

\bibliographystyle{alpha}
{\footnotesize\bibliography{CoarseSepHyp}}

\Address

\end{document}